\documentclass[11pt]{article}
\usepackage{amssymb,amsmath,amsthm}
\usepackage{amscd}
\usepackage[mathscr]{eucal}
\usepackage[english]{babel}
\usepackage[utf8]{inputenc}
\usepackage[title]{appendix}
\usepackage{graphicx}
\usepackage{graphics}
\usepackage{enumerate}
\newtheorem{theorem}{Theorem}

\usepackage{geometry}
\providecommand{\mod}{\mathop{\rm mod}\nolimits}
\newtheorem{lemma}{Lemma}

\newtheorem{remark}{Remark}
\newtheorem{assum}{Assumption}
\newtheorem{prop}{Proposition}[section]
\theoremstyle{definition}

\usepackage{hyperref}
\title{High-Frequency Two-Dimensional Asymptotic Standing Coastal-Trapped Waves in Nearly Integrable Case}
\author{A.~Yu.~Anikin$^1$\\ email: \href{anikin83@inbox.ru}{anikin83@inbox.ru} \and  V.~V.~Rykhlov$^{1,2}$ \\ email: \href{vladderq@gmail.com}{vladderq@gmail.com}}
\date{$^1$ Ishlinsky Institute for Problems in Mechanics RAS,\\
$^2$ Lomonosov Moscow State University}

\begin{document}
\maketitle

\section*{Abstract}
This paper is a~continuation of research started in \cite{ADNTs21} devoted to explicit asymptotic formulas for standing coastal-trapped waves. Our main goal is to construct formal asymptotic eigenfunctions of the wave operator $\langle \nabla, D(x)\nabla\rangle$ (the spatial part of the wave operator) corresponding to the eigenvalue $\omega\to\infty$ in a~nearly integrable case.

\section*{Introduction}
This paper is a~continuation of research started in~\cite{ADNTs20},~\cite{ADNTs21} devoted to explicit asymptotic formulas for standing coastal trapped waves in a~basin. The basin is a~domain $\Omega\subset\mathbb{R}^2$ given by $\mathcal D(x)>0$ with smooth boundary $\mathcal D(x)=0$, where the function $\mathcal D(x)$ (the depth function) is defined as follows:
\begin{equation}
\label{DefDx}
{\mathcal D}(x)\in C^\infty(\overline \Omega),\quad \text{and}\quad
\mathcal D\big|_{\Omega}>0,\quad \text{and}\quad \frac{\mathcal \partial \mathcal D}{\partial n}\Big|_{\partial \Omega} \ne 0,
\end{equation}
here $n$~is normal to the boundary (the coast) $\partial \Omega$.

Our goal is to construct the explicit form of leading terms for certain series of asymptotic eigenfunctions $\psi$ of the operator $\widehat{\mathcal{H}}$, given as follows
\begin{equation}
\label{EV-eq}
\widehat{\mathcal H}\psi=\mathbf{\omega}^2\psi,\qquad \widehat {\mathcal H}\psi:=- \langle \nabla, \mathcal D(x) \nabla \rangle\psi\equiv-\sum\limits_{j=1}^2\frac{\partial}{\partial x_j}{\mathcal D}(x)\frac{\partial}{\partial x_j}\psi
\end{equation}
as $\omega\to\infty$ (rigorous definitions are given below). These solutions describe waves that are, in a~certain way, analogs of Ursell waves (see \cite{Urs}, \cite{MerZ}, \cite{Zhev}).

The main difficulties with the operator $\widehat{\mathcal{H}}$ are caused by the fact that it degenerates at the boundary. For instance, the standard boundary conditions (e.g., Dirichlet or Neumann) lead to ill-posed problems, and it is reasonable to impose the finite energy condition (see \cite{OleyRadk}) instead of them. This degeneracy is also seen at the level of semiclassical asymptotics, as the conventional semiclassical methods based on quantizing Liouville tori (see, e.g., \cite{MasFed},\cite{Laz}) fail here. The reason is that the trajectories of the classical system associated with~$\widehat{\mathcal{H}}$ cease to exist as their positions approach the boundary (because the momenta go to infinity). As a~result, the classical system, given that it is integrable, folliates the phase space into a~family of noncompact Lagrangian manifolds instead of tori.

An approach to semiclassical asymptotic theory with such type of degeneracy was proposed in \cite{Naz12},\cite{Naz14}. The idea was to compactify the phase space near the boundary via special regularizing change of coordinates. 
By applying this approach, in \cite{ADNTs20}, \cite{ADNTs21}, a~certain class of domains and depth functions was found such that asymptotic solution of \eqref{EV-eq} can be written explicitly. Our goal is to extend these results by considering small perturbations of this case.

Thus, we are going to deal with a~close to integrable problem: the operator has a~principal symbol, which is a~completely integrable Hamiltonian, perturbed by some correction (the subprincipal symbol). The question of constructing asymptotic solutions for eigenfunctions of such operators is nontrivial even in the case without degeneracy. The fundamental work \cite{Laz} treats this question theoretically, though its approach does not lead to explicit formulas in specific problems. These ideas were developed in a~more practical manner in \cite{AnDob}, \cite{AniRyk22}, giving a~general recipe which is as follows. 

Among the family of Liouville tori, we choose an individual torus independent of the semiclassical parameter without bothering if it satisfies the Bohr--Sommerfeld quantization conditions. Next we find a~scalar function (called the amplitude) on the torus satisfying the transport equation (arising from the first approximation of the semiclassical theory). The torus is assumed to be such that this equation is solvable. A~sufficient condition for that is for the torus to be Diophantine. The final expression for the asymptotic eigenfunction is given in terms of the Maslov canonical operator, which is associated with that torus, applied to the amplitude function above times a~certain factor that is defined on the universal covering of the torus. 

Despite the fact that the canonical operator is ill-defined when the quantization conditions are violated, and is fed with a multi-valued function on the torus, these two obstacles (as sketched in \cite{AnDob}) balance each other. In this paper, we carefully state and prove this fact (in Appendix~\ref{pr-co-def}) using modern theory (\cite{DNSh17} as opposed to \cite{MasFed}) of the canonical operator based on phase functions.

To apply this method to our problem, we use a~recently developed alternative approach (see \cite{DN-Uni} and \cite{Unif23}) for semiclassical theory of operators with the degeneracy at the boundary. Its idea (called the uniformization) is to consider the operator as acting on the configuration space of the higher degree. Such a~lifted operator is free of degeneracy, and the conventional semiclassical theory with Liouville tori is applicable there.

We stress that the main purpose of our paper is to obtain the final formulas for asymptotic solutions in a~form as explicit as possible. Shortly speaking, they have the following form. Each solution is localized in some domain bounded by the caustics of two types: coastal caustic (a part of the boundary) or simple caustic (otherwise). The final formula is glued from representations in overlapping neighborhoods of different caustics. These representations are explicit expressions (except for the solution for the transport equation) in terms of Airy (for a~simple caustic) or Bessel (for a~coastal one) function. Needless to say that in the unperturbed way our formulas coincide with those from \cite{ADNTs20}, \cite{ADNTs21}.

The paper is organized as follows. In Section~\ref{Sec_stat}, we give a~formal statement of the problem. In Section~\ref{Sec_gen_th}, we prove theoretical results that the asymptotic solution can be represented in terms of the canonical operator. We start with a~simple case without coastal caustics, where the boundary degeneracy is not relevant. (In fact, this is a~perturbation of the case studied in \cite{Matv}.) Then we proceed to the case with a~shore by using the uniformization argument. In Section~\ref{Sec_exp_for}, we present an explicit representation for the asymptotic solution that were found earlier. In Section~\ref{Sec_ex}, we show some results of practical calculations.

\section{Problem Statement}\label{Sec_stat}
Consider a cylinder 
$$
\mathfrak C = \{u: u_{\rm L}<u< u_{\rm R}\} \times \mathbb{S}^1_v,\qquad
\text{where}\quad u_{\rm L}\in \mathbb{R}\cup\{-\infty\},\quad
u_{\rm R}\in \mathbb{R}\cup\{\infty\},
$$
and $\mathbb{S}^1_v$ is a~circle with the coordinate $v\:{\rm mod}\:2\pi$. Define an~operator on $L^2(\mathfrak{C})$, and the eigenproblem for it by
\begin{equation}\label{mainprobmu}
\widehat H = -h^2\langle \nabla_{u,v}, D_0(u,v)\mathbf{I}+\mu h \mathbf M_1(u,v) \nabla_{u,v}\rangle, \qquad  \widehat H\psi =\mathcal{E}\psi.
\end{equation}
Here $\mu$ is a~discrete parameter ($\mu = 0$ for the \textit{unperturbed} case, $\mu = 1$ for the \textit{perturbed} one),  $h \to 0+$ is a~small parameter; $D_0$ is a function, $\mathbf M_1$ is a~symmetric matrix, both smooth in $u, v$ and real-valued, moreover,
\begin{equation}
\label{Bound_eq}
D_0(u_{\rm L/R},v)=0,\quad \frac{\partial D_0}{\partial u}(u_{\rm L/R},v)\neq 0,\quad \mathbf M_1(u_{L/R},v)=0.
\end{equation}
In case $\mathfrak{C}$ is unbounded, we will always assume that $D_0$ and $\mathbf{M}_1$ at $u\to \pm\infty$ grow not faster than a polynomial together with all their derivatives. 

A~\textit{formal} asymptotic eigenfunction of~$\widehat H$ is a~function $\psi \in L^2(\mathfrak C)$ such that
$$
C_1\le \|\psi\|_{L^2(\mathfrak C)}\le C_2,\quad\Big\| \big(\widehat H - \mathcal E\big) \psi \Big\|_{L^2(\mathfrak C)} = O(h^2),
$$
where $0<C_1<C_2$.

In what follows we deal only with functions $D_0$ from a specific class, which we describe below. Now we show how problem \eqref{mainprobmu} arises from one initially stated in \eqref{DefDx}, \eqref{EV-eq}.

\subsection{Motivation}\label{Sec_Unp}

Let us consider $G=-h^2\langle \nabla_x,\mathbf{M}(x,h)\nabla_x\rangle$, where $\mathbf{M}$ is a~real matrix.
\begin{prop}[see \cite{ADNTs21}, Sec.~2.3]
\label{prop_conf} 
Let $x=X(y)$ be a change of coordinates such that the map $x_1+i x_2 \mapsto y_1 + i y_2$ is conformal, and $T:\psi(x)\mapsto (T\psi)(y)=\psi(X(y))|\det X'(y)|^{1/2}$. Then
\begin{equation}
\label{NewSymb}
T\widehat G T^{-1}=-h^2\langle \nabla_y,(X'(y))^{-1}\mathbf{M}(X(y),h)\big((X'(y))^{-1}\big)^{\rm T}\nabla_y\rangle+h^2a(x,h)
\end{equation}
for a smooth function $a(x,h)$. 
\end{prop}

Define the unperturbed domain as $\Omega_0:= \mathbf{r}(\mathfrak{C})$, where $\mathbf{r}:u+iv \to x_1+ix_2 = e^{u+iv}$. 

Consider a change of variables: $y=Y(x,h)\equiv x+h\xi(x,h)$. It maps $\Omega_0$ onto the family of the ''perturbed'' domains $\Omega_h\subset\mathbb{R}^2_y$. 

Now assume that for the perturbed domain problem \eqref{EV-eq} is set (of course, with variable $x$ replaced by $y$) with the depth function $\mathcal{D}=\mathcal{D}(y,h)$ satisfying \eqref{DefDx}. In other words, it is assumed that:
\begin{itemize}
\item
$\mathcal{D}(y,h)\equiv \mathcal{D}_0(y)+h\mathcal{D}_1(y,h) = 0$ at $\partial \Omega_h$.
\item
The function $\mathcal{D}_0(y)$ satisfies \eqref{DefDx} in $\Omega_0$.
\end{itemize}
It can be easily derived that condition $\frac{\partial D}{\partial n}(y,h) \neq 0$ at $\Omega_h$ is met automatically. 

Straightforward computations show that the operator $\widehat H = -\langle\nabla_y \mathcal{D}\nabla_y\rangle$ satisfies 
\begin{equation}
\label{SymbMatForm}
T\widehat HT^{-1}= -h^2\langle\nabla_x, \widetilde{\mathbf{M}} p \nabla_x\rangle + h^3 \langle Q(x), \nabla_x\rangle + O(h^4), \quad 
\end{equation}
where $\widetilde{\mathbf{M}}=D_0(x) \mathbf{I}+h \widetilde{ \mathbf{M}}_1(x)+O(h^2)$, and
$$
\widetilde{ \mathbf{M}}_1(x) = \big(D_1(x,0)+\frac{\partial D_0}{\partial x}(x)\xi(x,h)\big) \mathbf{I} - 2 D_0(x)\big(\frac{\partial \xi}{\partial x}\big)^{\rm T},\quad Q(x) = D_0(x) \begin{pmatrix} \Delta \xi_1 \\ \Delta \xi_2 \end{pmatrix},\quad \Delta = \frac{\partial^2}{\partial x_1^2}+\frac{\partial^2}{\partial x_2^2}.
$$
From this we obtain that $\widetilde{ \mathbf{M}}_1\big|_{\partial \Omega_0}=0$ if $\langle\frac{\partial D_0}{\partial x}(x),\xi(x,h)\rangle\big|_{\partial \Omega_0} = 0$. The symbol of the operator~$\widehat H$ is
$$
H = H_0 + h H_1,\quad \text{where}\quad H_0 = D_0 |p|^2,\quad H_1 = (D_1+\frac{\partial D_0}{\partial x}\xi)|p|^2 - 2 D_0 \langle p, \big(\frac{\partial \xi}{\partial x}\big)^{\rm T} p\rangle + \langle Q(x),p\rangle.
$$

\subsection{Unperturbed Problem}

Let us recall the class of domains and depth functions introduced in \cite{ADNTs21} that ensure the integrability of corresponding classical problem.

Define $D_0(u,v)=\frac{1}{f(u)+g(v)}$, where $g\in C^\infty(\mathbb{S}^1)$, $f \in C^\infty((u_{\rm L}, u_{\rm R}))$, $f>0$, and
$$
\lim\limits_{u\to u_{\rm L/R}} f(u)(u-u_{\rm L}) = M_{\rm L/R} \neq 0,
\qquad \text{if} \quad u_{\rm L/R} \quad \text{is finite}.
$$

Consider a~Hamiltonian system with the depth function~$D_0$ chosen as above (the main symbol of operator \eqref{mainprobmu}):
\begin{equation}
\label{main_Ham} 
H_0= \big(p_u^2+p_v^2\big)D_0(u,v)\equiv \frac{p_u^2+p_v^2}{f(u)+g(v)}\equiv \frac{f_2(u)\big(p_u^2+p_v^2\big)}{f_1(u)+f_2(u)g(v)}
\end{equation} 
in the phase space $T^*\mathfrak{C}=\mathfrak{C}\times \mathbb{R}^2$. This system is completely integrable, the first integral is $F=\frac{p_u^2g(v)-p_v^2f(u)}{f(u)+g(v)}$. Then the invariant manifolds
\begin{equation}
\label{MEk}
\Lambda=M_{E,\kappa} := \{(u,v,p_u,p_v)\vert H_0 = E, \, F = -\kappa\}
\end{equation}
are Lagrangian and each of them consists of two connected components
$$
\Lambda = \Lambda^+ \cup \Lambda^-,\qquad \text{where} \quad \Lambda^\pm \cong \Lambda_1 \times \Lambda^\pm_2,\quad \text{and}
$$

\begin{equation}
\label{Lagr0}
\Lambda_1 = \{ (u, p_u): \, p_u^2 = E f(u) - \kappa \}, \quad \Lambda_2 =\Lambda_2^+\cup\Lambda^-= \{ (v, p_v): \, p_v^2 = E g(v) + \kappa \}.
\end{equation}
Here $\Lambda_2^\pm$ are two connected components, corresponding to the values $\pm p_v >0$ respectively. Consider $\Lambda^+$, omitting everywhere below index $+$.

Lagrangian manifolds $M_{E,\kappa}$ give rise to a series of asymptotic eigenfunctions that are localized in $\mathfrak{C}_{E,\kappa}$ -- a projection of $M_{E,\kappa}$ onto the configuration space $\mathcal{C}$. 

We will distinguish between two cases. In the first one, 'without a shore': $\mathfrak{C}_{E,\kappa}\cap \partial\mathfrak{C}=\emptyset$, and $M_{E,\kappa}$ is a torus. In the second case, 'with a shore', which takes place otherwise, $M_{E,\kappa}$ is a cylinder. 
\section{General Theorem}\label{Sec_gen_th}

\subsection{Case without Shore}
To construct asymptotic eigenfunctions, we will use the method of the Maslov canonical operator \cite{MasFed}. In Appendix~\ref{pr-co-def}, we recall a~modern definition of the canonical operator, and give some generalizations that are of importance for us. 

In a~nutshell, an~asymptotic eigenfunction $\psi$ is generated by a~Lagrangian manifold $\Lambda$ from the family above by the formula: $\psi = K_{\Lambda}[A]$ for a specially chosen function $A\in C^\infty(\Lambda,\mathbb{C})$, where $K_{\Lambda}: C^\infty(\Lambda)\to C^\infty(\mathfrak{C}\times (0,1]_h)$ is the canonical operator on~$\Lambda$. The manifold $\Lambda$ is supposed to meet the quantization conditions, which have the following form. For $\nu=(\nu_1,\nu_2) \in\mathbb{Z}^2$, define {\it actions}:
\begin{equation}\label{BS0}
I^\nu_j:=\frac{1}{2 \pi} \oint\limits_{\gamma_j} p\,dx = h \Big(\nu_j + \frac{{\rm ind}_{\gamma_j}}{4}\Big)
\end{equation}
for $j=1,2$ with basis cycles $\gamma_j$ on the torus $\Lambda$, where ${\rm ind}_{\gamma_j}$ is the Maslov index of the cycle $\gamma_j$. By taking $\Lambda_{1,2}$ as basis cycles, and by denoting
\begin{equation}
\label{Actions0}
S_1(u_1,u_2)= \int\limits_{u_1}^{u_2} \sqrt{E f_1(\widetilde u)/f_2(\widetilde u) - \kappa} d\,\widetilde u, \quad S_2(v) = \int\limits_{0}^{v} \sqrt{E g(\widetilde v) + \kappa} \,d\widetilde v,
\end{equation}
and also $S_{\rm L}(u) =S_1(u_{\rm L}^0,u)$, $S_{\rm R}(u)=S_1(u,u_{\rm R}^0)$, the quantization conditions \eqref{BS0} can be written in an explicit form as follows
\begin{equation}
\label{BS1}
2\pi I_1\equiv 2 S_{\rm L}(u_{\rm R}^0)\equiv 2 S_{\rm R}(u_{\rm L}^0)=2\pi h(\nu_1+1/2),\qquad 2\pi I_2\equiv S_2(2\pi)=2\pi h\nu_2.
\end{equation}

In what follows, it will be convenient to parametrize the family $\Lambda = M_{E,\kappa}$ by the values of $I=(I_1,I_2)$. We will use the notation $\Lambda = \Lambda(I)$.

Asymptotic solution of unperturbed problem is given by $\psi=K_{\Lambda}[1]$. More precisely, $\Lambda$ must be endowed with a measure which is invariant with respect to the Hamiltonian flow. Let us take coordinates on $\Lambda$ that agree with the invariant measure. Namely, let $\alpha=(\alpha^1,\alpha^2)\mod 2\pi$) be coordinates such that the Hamiltonian flow on $\Lambda$ is given as:
\begin{equation}
\label{Gd_Crd}
\dot \alpha = \frac{\omega}{\mathbf{f(\alpha)}},
\end{equation}
where $\omega=\frac{\partial H}{\partial I}$ is the frequency vector, and $\mathbf{f}(\alpha)$ is some smooth nonvanishing function. (The invariant measure is then given as $d\mu = \frac{d\alpha}{\mathbf{f}}$.) Later on we will show that such coordinates can be constructed explicitly. (These coordinates were not constructed in \cite{ADNTs21} since they are essential only for the perturbed case).

In the perturbed case, following \cite{AnDob}, we fix an~invariant torus $\Lambda(I)$ for a~given $I$ independent of $h$. For each $\nu\in\mathbb{Z}^2$ define \textit{action defect} $q \equiv I^\nu - I$ (see \cite{AnDob, AniRyk22}; see also \textit{mismatch of quantization conditions} in \cite{LazQuas}). Define $e^{\frac{i \langle q, \alpha\rangle}h}$ as a function on the universal covering of the manifold $\Lambda(I)$. Although the quantization conditions on $\Lambda(I)$ are violated (for generic $h$), and $e^{\frac{i \langle q, \alpha\rangle}h}$ is not even a function on $\Lambda(I)$, the following statement holds true:
\begin{theorem}
Under the assumptions above,
$K_{\Lambda(I)}^{d\mu}[\overline A]$, where $\overline A=A e^{\frac{i \langle q, \alpha \rangle}{h}}$, is well-defined. 
\end{theorem}
\begin{proof}
The sketch of the proof was given in \cite{AnDob}. Here we give a complete proof based on the modern representation of the canonical operator from \cite{DNSh17}. See Proposition~\ref{CO-QC-gen} in Appendix~\ref{pr-co-def}.
\end{proof}

Let $\Lambda$ be some fixed torus. Denote by $H_{\rm sub}$ the ``subprincipal symbol'' of the operator, i.e. $H_{\rm sub} := (p_u,p_v)\mathbf{M}_1(u,v)(p_u,p_v)^{\rm T}$. Consider the following equation
\begin{equation}\label{transporteqdefect2}
\Big(\Big\langle \omega, \frac{\partial}{\partial \alpha} \Big\rangle + i \mathbf{f}(\alpha)(H_{\rm sub}\big|_{\Lambda(I)} - \lambda) + i \frac{\langle \omega, q \rangle}{h}\Big) A = 0
\end{equation}
(called the {\it transport equation}) with a~parameter $\lambda$, and
an unknown function $A(\alpha)\in C^\infty(\Lambda)$. Let us recall a standard fact:
\begin{lemma}
\label{tr-sol-lambda}
If vector $(\omega_1,\omega_2)$ is Diophantine, then there is a solution of \eqref{transporteqdefect2} for
$$
\lambda = \frac{\int_{\Lambda(I)} \big( \mathbf{f}(\alpha) H_{\rm sub}\big|_{\Lambda(I)}(\alpha) + \frac{\langle \omega, q \rangle}{h}  \big) d\alpha}{\int_{\Lambda(I)} \mathbf{f}(\alpha)d\alpha}.
$$
\end{lemma}

Here is the general statement describing a series of asymptotic eigenvalues and eigenfunctions in the case without a~shore:
\begin{theorem}
\label{PrAsSolA}
Let~$\alpha$ be coordinates on $\Lambda$ as indicated above. Suppose that $A\in C^\infty(\Lambda)$ and $\lambda\in \mathbb{R}$ is a~solution of \eqref{transporteqdefect2}.
Then
$$
\big\|\big(\widehat H-\mathcal{E}\big)\widetilde\psi\big\|=O(h^2),\qquad \text{where}
\quad \mathcal{E}=E+h\lambda,\quad \widetilde\psi=K^{d\alpha}_{\Lambda(I)} [A e^{\frac{i \langle q, \alpha \rangle}{h}}],
$$
uniformly with respect to $\nu$ such that $|I_\nu-I|\le C h$.
\end{theorem}
\begin{proof} 
Let us use the commutation formula (see \cite{MasFed}) for a~differential operator $\widehat H$ and the canonical operator $K_\Lambda$:
$$
(\widehat H - \widetilde E) K_{\Lambda(I)}[\overline A] = -ih K_{\Lambda(I)}[(v_{H_0}+i (H_{\rm sub}\big|_{\Lambda(I)}-\lambda))\overline A + O(h)],
$$
where $v_{H_0}$ is differentiation along the Hamiltonian vector field, i.e. $v_{H_0} \equiv \frac{d}{dt} = \langle \dot\alpha, \partial/\partial \alpha\rangle = \Big\langle \frac{\omega}{\mathbf f(\alpha)},\partial/\partial \alpha\Big\rangle$.
Applying the differential operator $\Big\langle \frac{\omega}{\mathbf f(\alpha)},\partial/\partial \alpha\Big\rangle$ to a function $\overline A = A e^{\frac{i\langle q, \alpha\rangle}{h}}$, we complete the proof.
\end{proof}
In what follows we give an explicit representation for $\widetilde \psi$ determined by Theorem~\ref{PrAsSolA}.
\subsection{Case with Shore}

Assume now that conditions B are satisfied. 

Let us introduce ``the lifted configuration space'': $\overline{\mathfrak{C}}:=(\mathfrak{C}\times\mathbb{S}^1)/\sim$, where $\sim$ glues the points of each circle $\{z\}\times\mathbb{S}^1$ for each fixed $z\in\partial\mathfrak{C}$. We will endow $\overline{\mathfrak{C}}$ with local coordinates $(u,v,\varphi)$ ``away from the shore'', where $(u,v)\in\mathfrak{C}$, and $\varphi\,{\rm mod}\,2\pi$ for $u\in(u_{\rm L}+a,u_{\rm R}-a)$ (with arbitrary $a> 0$), and local coordinates $(y_1,y_2,v)$ ``near the shore'' with $y=(y_1,y_2)$ determining points of the disc $|y|\le b$. The coordinates are glued via the relations: $(y_1,y_2)=2\sqrt{|u-u_{\rm L/R}|}\mathbf{n}(\varphi)$, where $\mathbf{n}(\varphi)=(\cos\varphi,\sin\varphi)$. On $\overline{\mathfrak{C}}$ there is a natural semi-free action of group $\mathbb{S}^1$ given by $\varphi\mapsto \varphi + \beta$ for $\beta\,{\rm mod}\,2\pi$ away from the shore, and by rotations around origin on the $y$-plane near the shore. Then $\mathfrak{C}$ can be identified with the quotient space of the orbits under the group action: $\mathfrak{C}=\overline{\mathfrak{C}}/\mathbb{S}^1$, and space $C^\infty_{\mathbb{S}^1}(\overline{\mathfrak{C}})$ consisting of functions $\psi\in C^\infty(\overline{\mathfrak{C}})$ that are invariant under action of the group $\mathbb{S}^1$ can be identified with $C^\infty(\mathfrak{C})$. 

Define the following operator acting on $\psi=\psi(y,v)$ supported near the shore: 
\begin{equation}
\label{op_lift_sh}
\widehat{\mathscr H}_{\rm Sh} = h^2 \Big( \Big\langle Y, \mathbf Q(y, v) Y \Big\rangle + \big \langle Z, \mathbf Q(y, v) Z \Big\rangle \Big),
\qquad
\mathbf Q=\frac{4}{y^2}\bigg(D_0\Big(\frac{y^2}{4},v\Big)\mathbf I + h \mathbf M_1 \Big(\frac{y^2}{4},v\Big) \bigg),
\end{equation}
where $Y_1 = -i \frac{\partial}{\partial y_1}$, $Y_2 = -i \frac{y_1}{2}\frac{\partial}{\partial v}$, $Z_1 = - i \frac{\partial}{\partial y_2}$, $Z_2 = -i \frac{y_2}{2} \frac{\partial}{\partial v}$. It follows from Theorem 2.1 in \cite{Unif23} that $\widehat{\mathscr H}_{\rm Sh} \psi= \widehat H \psi$ (where $\widehat H$ is given by \eqref{mainprobmu}), if $\psi\in C^\infty_{\mathbb{S}^1}(\overline{\mathfrak{C}})$. This allows us to define $\widehat {\mathscr{H}}:C^\infty_{\mathbb{S}^1}(\mathfrak{C})\to C^\infty_{\mathbb{S}^1}(\mathfrak{C})$ by setting $\widehat{\mathscr{H}}\psi:=\widehat{\mathscr{H}}_{\rm Sh}\psi$ (respectively, $\widehat{\mathscr{H}}\psi:=\widehat H\psi$) for $\psi$ supported near the shore (respectively, away from the shore).

Consider the standard phase space $T^*\overline{\mathfrak{C}}$ with coordinates $(y,\xi,v,p_v)$ near the shore, and $(u,p_u,v,p_v,\varphi,p_{\varphi})$ away from the shore, where $\xi$ (respectively, $p_\varphi$) are momenta conjugated to $y$ (respectively, $\varphi$). Define a Hamiltonian system given by the symbol of operator $\widehat{\mathscr{H}}$, i.e.
\begin{equation}\label{lifted_ham_y}
\mathscr H_{\rm Sh} = \frac{4}{y^2}\Big(\xi^2 + \frac{y^2}{4} p_v^2\Big) D_0\Big(\frac{y^2}{4},v\Big)=\frac{\xi^2+\frac{y^2}{4}p_v^2,}{f_1\big(y^2/4\big)+\frac{y^2}{4}g(v)}
\end{equation}
(which is a symbol of \eqref{op_lift_sh}) near the shore, and by $\mathscr{H}=H_0$ from \eqref{main_Ham} away from the shore. Of course, two Hamiltonian systems \eqref{lifted_ham_y} and \eqref{main_Ham} are not identical as are not the operators $\widehat{\mathscr{H}}_{\rm Sh}$ and $\widehat H$. However, it can be seen that $\mathscr{H}_{\rm Sh}-H_0$ is proportional to $p_\varphi^2$, and thus the Hamiltonian vector fields $V_{\mathscr{H}_{\rm Sh}}$ and $V_{H_0}$ restricted to the hypersurface $\Sigma:=\{p_\varphi=0\}\equiv\{y_1\xi_2=y_2\xi_1\}$ coincide. 

Let us show that vector field $V$ obtained by restriction of $V_{\mathscr{H}_{\rm Sh}}=V_{H_0}$ on $\Sigma$ is completely integrable, meaning that $\Sigma$ is foliated into a two-parametric family of $3$-tori invariant under $V$. Indeed, two integrals in involution $F_1$ and $F_2$ can be given as follows: 
$$
F_1=p_\varphi, \quad F_2 = \frac{p_u^2g(v)f_2(u)-p_v^2f_1(u)}{f_1(u)+f_2(u)g(v)}\qquad \text{\bigg(respectively,}\quad F_1=y_1\xi_2-y_2\xi_1,
\quad F_2 = \frac{\xi^2 g(v)- p_v^2 f_1(\frac{y^2}{4})}{f_1(\frac{y^2}{4})+\frac{y^2}{4}g(v)}\bigg)
$$
away from the shore (respectively, near the shore). It is clear that two representations of $F_1$ define the same function. As to $F_2$, the difference of two representations (again by Theorem 2.1 from \cite{Unif23}) is proportional to $p_\varphi^2$. Next, it is clear that $H_0$, $F_1$ and $F_2$ in the representation away from the shore are Poisson commuting. Hence, equations $H_0=E, F_1=0, F_2=-\kappa$ determine a well-defined manifold denoted as $\Lambda_{E,\kappa}\subset T^*\mathfrak{C}$, which is Lagrangian and invariant under $V_{H_0}$, $V_{F_1}$ and $V_{F_2}$, moreover, these vector fields are commuting on $\Lambda_{E,\kappa}$. 

Manifold $\Lambda_{E,\kappa}$ is a 3-torus, parametrized by:
$$
\xi^2 = E f_1(y^2/4)-\kappa y^2/4,\quad p_v^2=Eg(v)+\kappa,\quad y_1\xi_2=y_2\xi_1
$$
near the shore, and
$$
p_u^2 = E f_1(u)-\kappa f_2(u), \quad p_v^2=Eg(v)+\kappa,\quad p_\varphi=0
$$
away from the shore. Let us choose the basis of cycles as follows: $\gamma_1$ by fixing all variables except for $u,p_u$; $\gamma_2$ -- except for $v,p_v$, and $\gamma_3$ -- except for $\varphi,p_\varphi$. Then the quantization conditions \eqref{BS0} for cycles $\gamma_1$ and $\gamma_2$ coincide exactly with \eqref{BS1}. The quantization condition for $\gamma_3$ is met automatically.

The solution of the unperturbed problem is then written as $\psi = K_{\Lambda}1$, where $\Lambda=\Lambda_{E,\kappa}$ satisfies \eqref{BS1}.

The perturbed problem is solved in a complete analogy with the case without a~shore. Namely, let us express $H_0$ as a~function of actions $I_1$, $I_2$ (defined in \eqref{Actions0}), and denote by $\omega = \frac{\partial H_0}{\partial I}$ the frequency vector. Let us choose an individual torus $\Lambda=\Lambda_{E,\kappa}$ corresponding to actions $ I_1, I_2$, and endow it with coordinates $\alpha^1,\alpha^2,\varphi$ (mod $2\pi$) such that the Hamiltonian vector field is given as $\dot\alpha=\frac{\omega}{\mathbf{f}(\alpha)}$ and $\dot\varphi=0$.

The subprincipal symbol of the operator $\widehat {\mathscr{H}}$ restricted on $\Lambda$ is also well-defined and given by
$$
H_{\rm sub}=\frac{4}{y^2}(\xi_1\:\: \frac{y_1p_v}{2})\mathbf{M}_1\big(y^2/4,v\big)\bigg(
\begin{aligned} \xi_2\\
\frac{y_2p_v}{2}
\end{aligned}
\bigg)\qquad \text{\Big(respectively,}\quad H_{\rm sub}=(p_u\:\:p_v)\mathbf{M}_1\Big(
\begin{aligned}&p_u\\&p_v\end{aligned}\Big)\Big)
$$
near the shore (respectively, away from the shore). Note that thanks to the third equality in \eqref{Bound_eq} $H_{\rm sub}$ is smooth everywhere on $\Lambda$.

Let us set $\bar q_\nu := I_\nu - I $ for $\nu\in\mathbb{Z}^2$ (as earlier, $I_\nu$ are values of actions satisfying \eqref{BS1}), and let $A(\alpha,\varphi)=A(\alpha)$ be a~smooth function on~$\Lambda$ satisfying the same transport equation \eqref{transporteqdefect2} as before.

As earlier (see Lemma~\ref{tr-sol-lambda}), the equation is solvable if $2$-D vector $\omega(\overline{I})$ is Diophantine, with $\lambda$ given by the same formula (Lemma~\ref{tr-sol-lambda}).
Then Proposition~\ref{CO-QC-gen} implies the following theorem:
\begin{theorem}\label{propcaseB}
Let a~function $A(\alpha)$ and a~number $\lambda$ give a~solution of \eqref{transporteqdefect2}. Then the expression $K_{\Lambda}^{\frac{d\alpha\wedge d\varphi}{\mathbf{f}}}\Big[Ae^{\frac{i\langle q,\alpha\rangle}{h}}\Big]$ is well-defined, and 
$$
\Big\|(\widehat H - \mathcal E)\widetilde \psi\Big\|_{L^2({\mathfrak C})}=O(h^2),\qquad  \widetilde \psi=K_{\bar \Lambda}[\overline A],\quad \mathcal E=E+h\lambda.
$$
\end{theorem}

\section{Explicit Formulas}\label{Sec_exp_for}
\subsection{Explicit Formulas for Angle Coordinates on Torus}\label{Alpha_formulas}

Dynamics on $\Lambda$ for coordinates $u, v$ is given by
\begin{equation}\label{dynamicaleqs}
\dot u = \frac{\partial H}{\partial p_u}\equiv \pm\frac{2\sqrt{Ef(u)-\kappa} }{f(u)+g(v)}, \quad \dot v = \frac{\partial H}{\partial p_v}\equiv  \frac{2 \sqrt{Eg(v)+\kappa}}{f(u)+g(v)},
\end{equation}
or, substituting $f(u) = f_1(u)/f_2(u)$,
\begin{equation}\label{dynamicaleqsf1f2}
\dot u = \pm\frac{2\sqrt{Ef_1(u)-\kappa f_2(u)} \sqrt{f_2(u)} }{f_1(u)+f_2(u) g(v)}, \quad \dot v = \frac{2 \sqrt{Eg(v)+\kappa}}{f_1(u)+f_2(u)g(v)} f_2(u).
\end{equation}
It is necessary to note that for all three cases $f_1(u)+f_2(u)g(v)$ is a smooth function on $\mathfrak C$ and it does not vanish on the nonstandard caustics in the case with shore.
Take out the common denominator by the change $t\mapsto\tau$, $\frac{d\tau}{d t} = \frac{1}{f_1(u)+f_2(u)g(v)}$. Thus,
$$
\frac{d u}{d\tau} = \pm 2 \sqrt{E f_1(u)-\kappa f_2(u)}\sqrt{f_2(u)}, \quad \frac{d v}{d\tau} = 2 \sqrt{E g(v) + \kappa} f_2(u).
$$
To be certain, let the boundary of domain be $u = u_{\rm L}$ and $u = u_{\rm R}$.
Define functions $\alpha^1$, $\alpha^2$ (we show below that these functions are the variables parametrizing $\Lambda$, which we are looking for) by
\begin{equation}\label{alphas}
\begin{gathered}
\alpha^1(u) = \frac{\pi}{T_1}\int\limits_{u_{\rm L}}^u \frac{d \, \tilde u}{\sqrt{E f_1(\tilde u) - \kappa f_2(\tilde u)}\sqrt{f_2(\tilde u)}}, \quad \alpha^2(u,v) = \frac{2\pi \eta(v)}{\eta(2\pi)} + \frac{2 T_1}{\eta(2\pi)}\int\limits_0^{\alpha^1(u)} \Big( \langle f_2 \rangle_{\alpha^1} - f_2(u(\tilde \alpha^1)) \Big) d\, \tilde \alpha^1,
\end{gathered}
\end{equation}
where $T_1 = \int\limits_{u_{\rm L}}^{u_{\rm R}} \frac{d \tilde u}{\sqrt{E f_1(\tilde u) - \kappa f_2(\tilde u)}\sqrt{f_2(\tilde u)}}$, $\eta(v) = \int\limits_0^v \frac{d\tilde v}{\sqrt{E g(\tilde v)+\kappa}}$, $\langle f_2 \rangle_{\alpha^1} = \frac{1}{2\pi} \int\limits_0^{2\pi} f_2(u(\alpha^1)) d\alpha^1$ is an averaged by $\alpha^1$ function $f_2$.
Define $T_2 = \frac{\eta(v)}{2 \langle f_2 \rangle_{\alpha^1}}$ as well. The values $\alpha^1 = 0$ and $\alpha^1 = \pi$ correspond to the caustics $u = u_{\rm L}$ and $u = u_{\rm R}$ respectively. $\alpha^1\in (0,\pi)$ corresponds to $p_u>0$, and we can continue $\alpha^1$ to $p_u<0$ by symmetry.

\subsection{Main Theorem}
Let us obtain an explicit representation for the asymptotic eigenfunctions from Theorem~\ref{PrAsSolA} using explicit formulas for $K_\Lambda{A}$ from  \cite{ADNTs19}, \cite{AniRyk22} in a~neighborhood of a~simple caustic, and also the formulas from~\cite{Unif23} and Appendix~\ref{pr-co-def} for coastal caustic. Define functions
\begin{equation}\label{defect_mult}
B^{L}_\pm(u) = \exp \Big(\frac{i}{h}q_1(\pm \alpha^1(u) - \pi)\Big), \qquad B^{R}_\pm(u) = \exp \Big(\pm \frac{i}{h}q_1(\alpha^1(u) - \pi) \Big)
\end{equation} 
and their following combinations (with the arguments $u$ and $v$ omitted):
\begin{equation}
\begin{aligned}
& A_{\rm ev}^{L/R}= \frac{1}{2} \Big(B^{L/R}_+ A(\alpha^1, \alpha^2)+B^{L/R}_- A(-\alpha^1, \alpha^2) \Big) \exp{(\frac{i}{h} q_2 \alpha^2)},
\\ & A_{\rm odd}^{L/R} = \frac{1}{2} \Big(-B^{L/R}_+ A(\alpha^1, \alpha^2)+B^{L/R}_- A(-\alpha^1, \alpha^2) \Big) \exp{(\frac{i}{h} q_2 \alpha^2)}.
\end{aligned}
\end{equation}

Let us first formulate the assumptions under which the theorem is valid.
\begin{assum}\label{assumQC}
Let torus~$\Lambda(I)$ satisfy quantization conditions~\eqref{BS1}.
\end{assum}
\begin{assum}
Let the vector of frequencies~$(\omega_1,\omega_2)$ (see~Sec.~\ref{Alpha_formulas}) be Diophantine, i.e.
$$
\exists C_j>0 \text{ such that } |\omega_1 k_1 + \omega_2 k_2| \ge \frac{C_1}{(|k_1|+|k_2|)^{C_2}} \quad \forall (k_1,k_2)\in\mathbb Z^2\setminus (0,0).
$$
\end{assum}
\begin{assum}
Let an~action defect $q = I^\nu - I$ satisfy the estimate $q = O(h)$ uniformly in $\nu \in \mathcal I\subset \mathbb Z^2$.
\end{assum}
\begin{assum}\label{assumTrEqSol}
Let~$A(\alpha)\in C^\infty(\Lambda)$, $\lambda\in\mathbb R$ (see~Lemma~\ref{tr-sol-lambda}) be a~solution of the transport equation~\eqref{transporteqdefect2}.
\end{assum}
\begin{assum}\label{assumCases}
\begin{enumerate}[(A)] 
\item $f(u)$ is bounded near $u=u_{\rm L/R}$.
\item $f(u) \to \infty$ as $u\to u_{\rm L/R}$.
\item $f(u) \to \infty$ as $u\to u_{\rm L}$ and $f(u)$ is bounded near $u=u_{\rm R}$.
\end{enumerate}
\end{assum}
Further, introduce a~partition of unity $\rho_{\rm L}(u)+\rho_{\rm R}(u)\equiv 1$ on $[u_{\rm L}, u_{\rm R}]$ such that support of $\rho_{\rm L/R}$ lies outside $u=u_{\rm R/L}$.

\begin{theorem}
Let assumptions~\ref{assumQC}--\ref{assumTrEqSol} and one of assumptions~\ref{assumCases} (A), \ref{assumCases} (B) or \ref{assumCases} (C) hold.  Then $\psi_\nu = K_{\bar \Lambda}[A\,e^{\frac{i \langle \bar q_\nu, \alpha \rangle}{h}}]$, $\nu\in\mathcal{I}$, is a~series of asymptotic solutions of problem \eqref{mainprobmu} with asymptotic eigenvalues $\mathcal E_\nu = E_\nu + h \lambda_\nu$, i.e. the bound $\| (\widehat H-\mathcal E_\nu)\psi_\nu\| = O(h^2)$ is uniform by $\nu\in\mathcal{I}$, and the leading term of each asymptotic solution admits a~representation
\begin{equation}\label{AiAiprimerepr}
\psi = \frac{\sqrt{\pi (f(u)+g(v))}}{\sqrt{2 h \sqrt{E f(u)-\kappa}\sqrt{E g(v)+\kappa}}} e^{\frac{i}{h}S_2(v)} \Big( \rho_{\rm L}(u) \overline{W_{\rm Ai}^{\rm L}(u,v,h)} +  e^{\frac{i}{h}S_L(u_R)-\frac{\pi i}{2}} \rho_{\rm R}(u) W_{\rm Ai}^{\rm R}(u,v,h) \Big)
\end{equation}
in case (A), 
\begin{equation}\label{JJrepr}
\psi = \frac{\sqrt{\pi (f(u)+g(v))}}{\sqrt{2 h \sqrt{E f(u)-\kappa}\sqrt{E g(v)+\kappa}}} e^{\frac{i}{h}S_2(v)} \Big( \rho_{\rm L}(u) \overline{W_{\mathbf J}^{\rm L}(u,v,h)} +  e^{\frac{i}{h}S_L(u_R)-\frac{\pi i}{2}} \rho_{\rm R}(u) W_{\mathbf J}^{\rm R}(u,v,h) \Big)
\end{equation}
in case (B), and, finally,
\begin{equation}\label{JAirepr}
\psi = \frac{\sqrt{\pi (f(u)+g(v))}}{\sqrt{2 h \sqrt{E f(u)-\kappa}\sqrt{E g(v)+\kappa}}} e^{\frac{i}{h}S_2(v)} \Big( \rho_{\rm L}(u) \overline{W_{\mathbf J}^{\rm L}(u,v,h)} +  e^{\frac{i}{h}S_L(u_R)-\frac{\pi i}{2}} \rho_{\rm R}(u) W_{\rm Ai}^{\rm R}(u,v,h) \Big)
\end{equation}
in case (C),
where
$$
\begin{gathered}
W_{\rm Ai}^{\rm L/R}(u,v,h) =  A_{\rm ev}^{\rm L/R}(u,v) (\frac{3 S_{\rm L/R}(u)}{2 h})^{1/6}  {\rm Ai}(-(\frac{3 S_{\rm L/R}(u)}{2 h})^{2/3})  +  i A_{\rm odd}^{\rm L/R}(u,v) (\frac{3 S_{\rm L/R}(u)}{2 h})^{-1/6} {\rm Ai}'(-(\frac{3 S_{\rm L/R}(u)}{2 h})^{2/3}),\\
W_{\mathbf J}^{\rm L/R}(u,v,h) = \sqrt{S_{\rm L/R}(u)} \big( A_{\rm ev}^{\rm L/R}(u,v) {\mathbf J}_0\big( \frac{S_{\rm L/R}(u)}{h} \big) + i A_{\rm odd}^{\rm L/R}(u,v) {\mathbf J}_1\big( \frac{S_{\rm L/R}(u)}{h} \big).
\end{gathered}
$$
\end{theorem}
\begin{proof}
For case~(A), see~\cite{AniRyk22}. For cases~(B) and~(C), see~Appendix~\ref{pr-co-def}.
\end{proof}

\subsection{Commentaries}
Note that our method allows one to fix a single Diophantine torus $\overline\Lambda$, and then consider a family of ``quantized'' tori in its $O(h)$-neighborhood  (on 2-dimensional plane of action variables). An important fact is that we can construct the entire spectral series by solving the transport equation only once. It is only necessary to recompute the number $\lambda$, since the action defect may change for different tori, but this is not that difficult.

Another advantage of our method is that in practice there is no need to precisely select the parameters $E$ and $\kappa$ in order to make the quantization conditions satisfied, an error of order $h$ is acceptable.

Often the condition for a torus to be Diophantine can be weakened: it is enough to require that the frequencies do not have resonances for such~$k$ that~$L_k \ne 0$. It seems appropriate to call such tori \textit{nonresonant}.

Considering a~nonperturbed problem and setting the action defect to zero, we obtain the formulas which coincide exactly with \cite{ADNTs21}.

Note also that when using this approach in practice, it is enough to compute only a~few Fourier coefficients $L_k$, $|k|\le N$, since the bound
$$
\Big\| (H_{\rm sub}-\lambda)\mathbf{f}(\alpha)  - \sum\limits_{|k|\le N} L_k e^{i \langle k, \alpha \rangle} \Big\| = O\Big(\frac{1}{N^w}\Big)
$$
holds. Then the bound $\| (\widehat H-\mathcal E_\nu)\psi_\nu\| = O(h^2)$ from the theorem becomes $\| (\widehat H-\mathcal E_\nu)\psi_\nu\| = O(h^2 + \frac{1}{N^{\widetilde w}})$, which is fine for practical use, considering~$\frac{1}{N^w}$ to be of the same order as~$h$ to some power $\ge 2$.  

We should also mention that although the problem has physical interpretation and the desired solution represents an~elevation of the free surface, the solution we obtain is a~complex-valued function. But there is nothing to worry about, as its real and imaginary parts are also solutions.

\section{Examples}\label{Sec_ex}

\subsection{Example 1: Away from the Shore}

Despite the apparent complexity, asymptotic eigenfunctions computed by formula \eqref{AiAiprimerepr} can be implemented, e.g.,  in Wolfram Mathematica or Maple.

Consider the following example. Let $f(u) = e^{u (2.7 - u)} - 1.03$, $g(v) = \frac{4}{5} + \frac{1}{3}\cos{(3v)}\sin^2{(v)}$, $D_1(u,v) = e^{\sin{(3u)}} \cos^2{(2v)}$, $E = 1$, $\kappa = -0.03$. The equation $E f(u) = \kappa$ has roots $u_{\rm L}^* = 0$ and $u_{\rm R}^* = 2.7$. Consider the problem for $\omega = 14.3$ and ''quantum numbers'' $\nu = (18, 12)$. Then $\bar q_\nu / h = (-0.8, -0.492)$, and the frequencies $(\omega_1, \omega_2) \approx (2.533, 1.7306)$ have only high resonances. 

Profile of bottom line for $v=0$ is depicted in Fig.~\ref{fig:exAbottom}.
Plot of leading term of asymptotics \eqref{AiAiprimerepr} is in Fig.~\ref{fig:exampleAiAi}.


\begin{figure}[h!]
\begin{minipage}{0.47\textwidth}\centering
\includegraphics[width=0.93\textwidth]{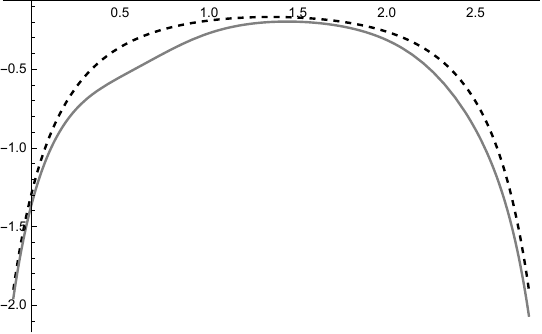}
\caption{Profile of bottom line for $v=0$. Black dashed line is bottom without perturbation, gray solid line is perturbed bottom.}
\label{fig:exAbottom}
\end{minipage}
\begin{minipage}{0.47\textwidth}\centering
\includegraphics[width=0.95\textwidth]{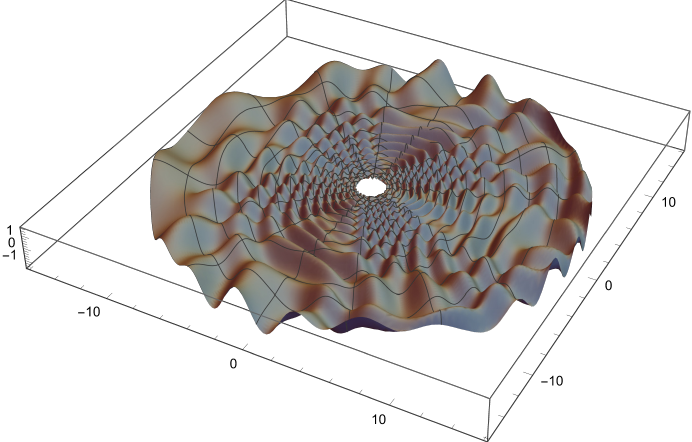}
\caption{Leading term of asymptotics $\psi$ by formula \eqref{AiAiprimerepr}.}
\label{fig:exampleAiAi}
\end{minipage}
\end{figure}


\subsection{Example 2: One Coastal Caustic}
Consider the problem with the following parameters:
$$
\begin{gathered}
f(u) = \frac{2}{3u} e^{-(u-\sqrt{2})^2} + \frac{2}{3} \sin{u}, \quad g(v) = 1 +\frac{2}{3} \sin{v},\quad E = 1,\quad a = f(2),\\ \omega = 17.8,\quad \nu = (10,18),\quad D_1(u,v) = \frac{2}{23} (2-u)^4 (e^u -1)^4 e^{\sin{2v}}.
\end{gathered}
$$
Then $u_{\rm L} = 0$, $u_{\rm R}^0 = 2$,  $h \approx 0.0562$, $\bar q_\nu/h \approx (0.0834, -0.3066)$, $(\omega_1,\omega_2) \approx (1.9955,1.94355)$. The profile of the bottom line for $v=0$ is depicted in Fig.~\ref{fig:exC-bottom}, and a~solution constructed using~\eqref{JAirepr} is depicted in Fig.~\ref{fig:caseC}.
\begin{figure}[h!]
\begin{minipage}{0.42\textwidth}
\includegraphics[width=0.78\textwidth]{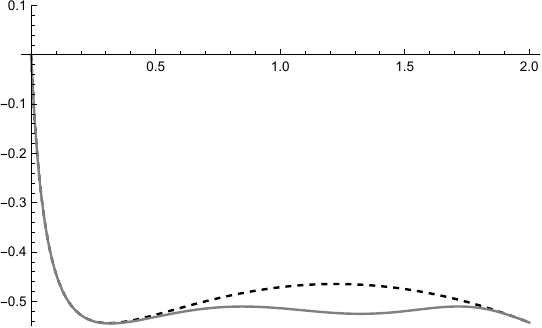}
\caption{The profile of the bottom line for $v=0$. Black dashed line is the unperturbed bottom, the gray solid is the perturbed bottom.}
\label{fig:exC-bottom}
\end{minipage}\hfill
\begin{minipage}{0.42\textwidth}
\centering
\includegraphics[width=0.78\textwidth]{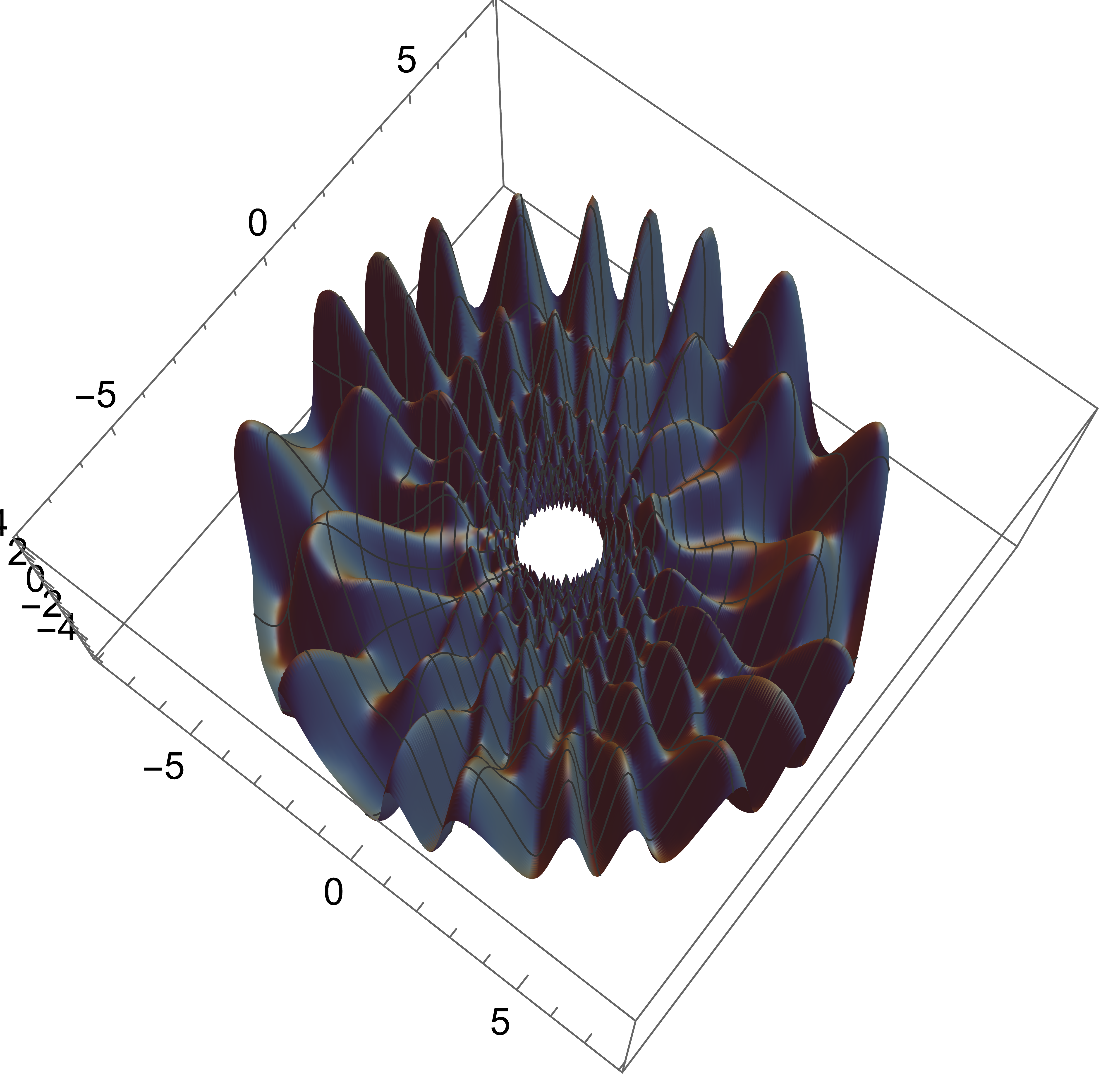}
\caption{An asymptotic solution $\psi_\nu$.}
\label{fig:caseC}
\end{minipage}
\end{figure}

\subsection{Example 3: Two Coastal Caustics}
Consider the further example:
\begin{equation}\label{main_example}
f(u) = \frac{1}{2u(1-u)}, \quad g(v)=3+\frac{1}{2}\sin{v}, \quad x_1 = e^u\cos{v}, \quad x_2 = e^u\sin{v},
\end{equation}
with $E=1$, $D_1(u,v) = 39 (1-u)^4 (e^u-1)^4 e^{\sin{3v}}$.
Singularities appear at the boundary defined by equations $u=0$ and $u=1$.
Phase trajectories of $H_0$ on the plane $(u,p_u)$ are depicted in Fig.~\ref{fig:stream}.

Consider the following sets of parameters:
\begin{enumerate}
\item $\kappa = -2.132799706586304$, $\omega = 12.08646478547537$, and $\nu = (10,11)$.
\item $\kappa = -2.029449909118645$, $\omega = 33.0692001800103$, and $\nu = (28,32)$.
\end{enumerate}

The profiles of the bottom for $v=0$ are depicted in Fig.~\ref{fig:exB-bottom}.

\begin{figure}[h!]
\begin{minipage}{0.45\textwidth}
\includegraphics[width=0.6\textwidth]{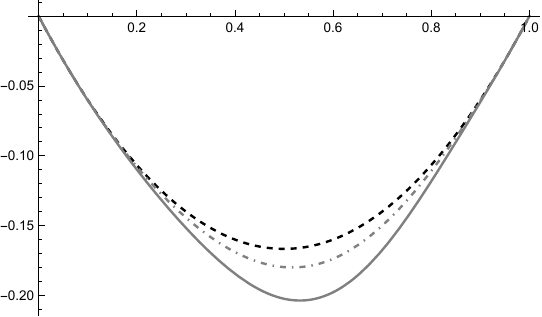}
\caption{The profiles of the bottom for $v=0$. Black dashed line is the unperturbed bottom, the gray solid one is the perturbed bottom (case 1), the gray dash-dotted one is the perturbed bottom (case 2).}
\label{fig:exB-bottom}
\end{minipage}\hfill
\begin{minipage}{0.45\textwidth}
\centering
\begin{center}
\includegraphics[width=0.6\textwidth]{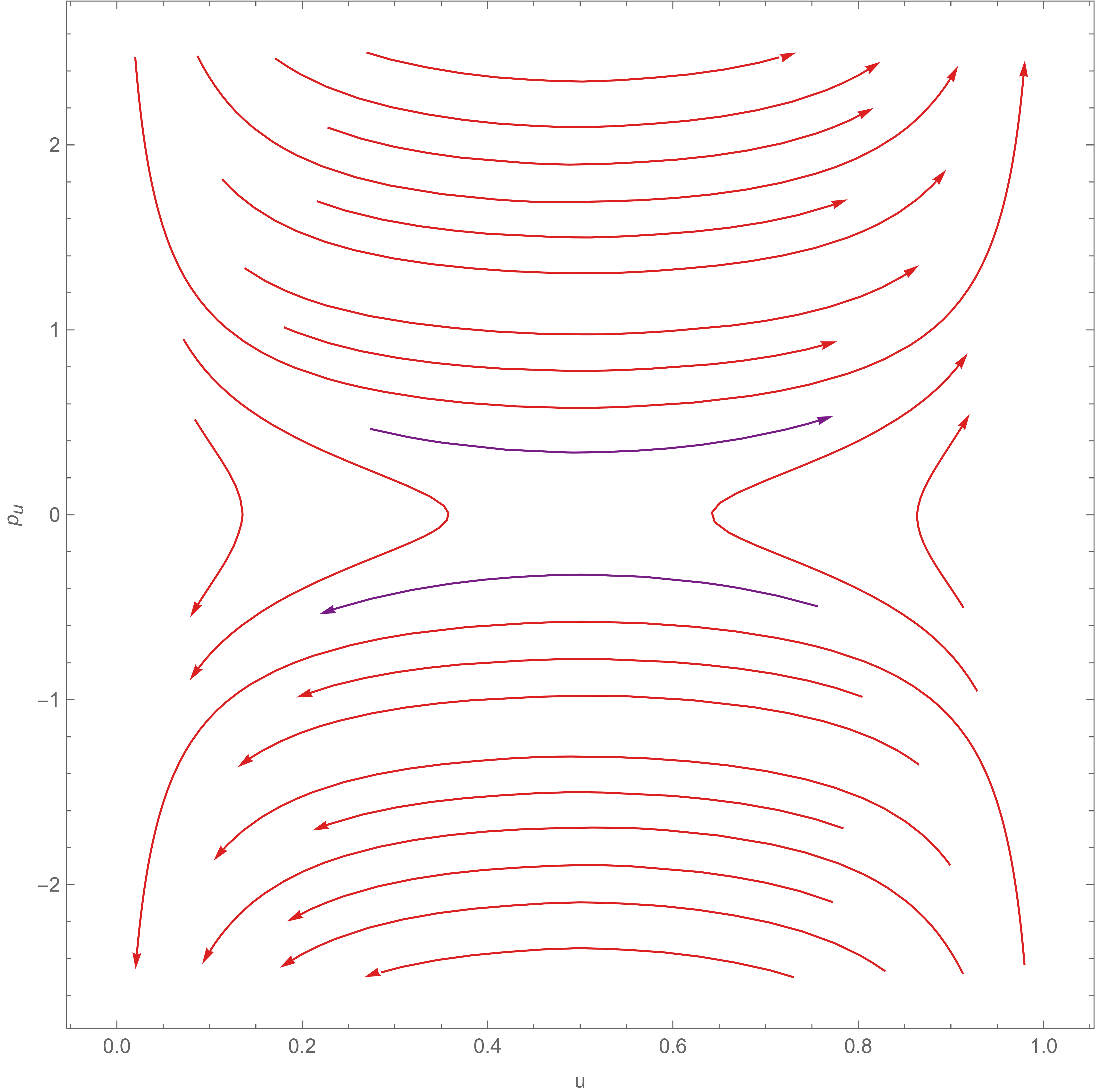}
\end{center}
\caption{The phase trajectories of $H_0 = \frac{p_u^2+p_v^2}{f(u)+g(v)}$.}
\label{fig:stream}
\end{minipage}
\end{figure}

The plots of the corresponding solutions $\psi_{\nu}$ are depicted in Fig.~\ref{fig:case1} and Fig.~\ref{fig:case2}.

\begin{figure}[h!]
\begin{minipage}{0.4\textwidth}
\centering
\includegraphics[width=0.78\textwidth]{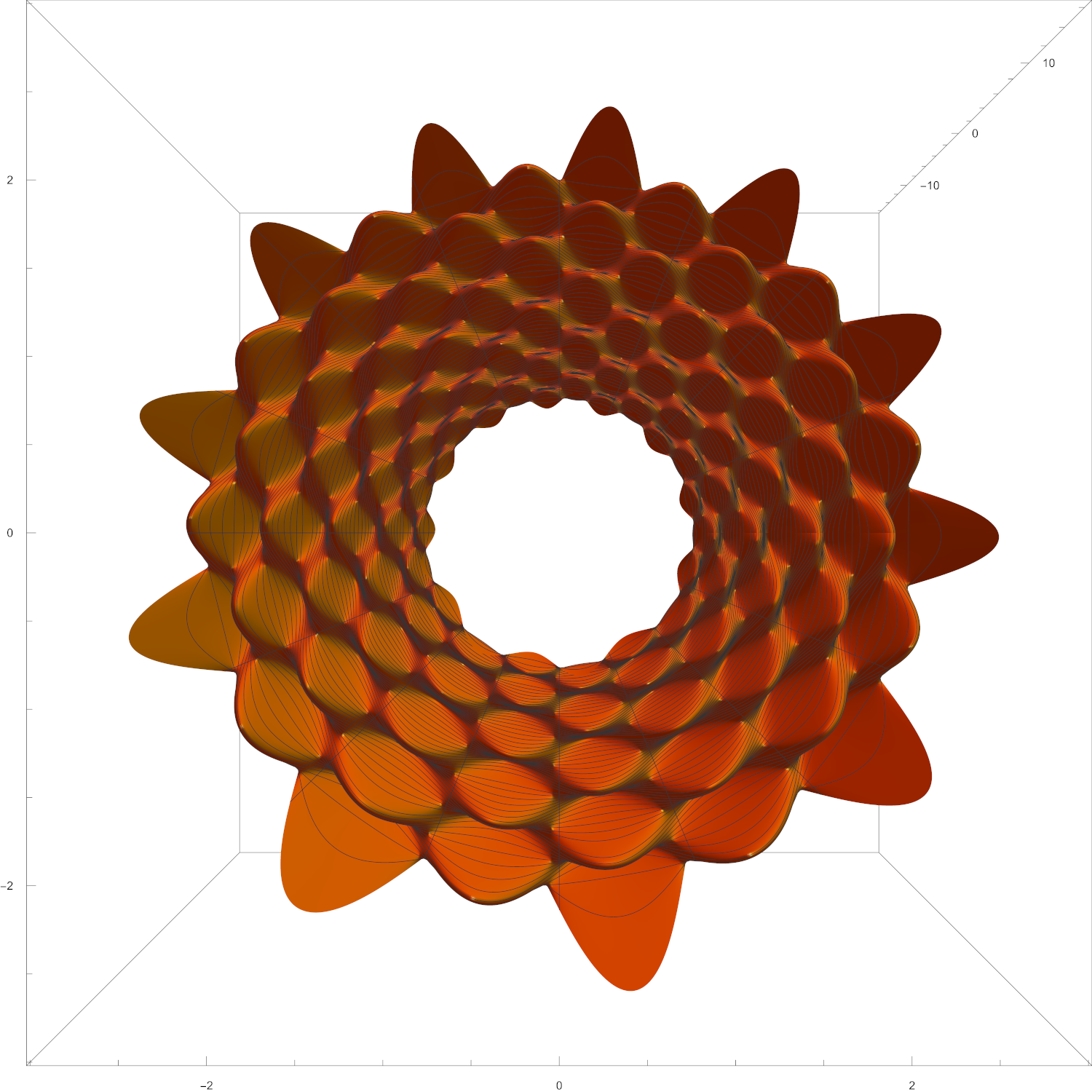}
\caption{Case 1.}
\label{fig:case1}
\end{minipage}
\begin{minipage}{0.4\textwidth}
\centering
\includegraphics[width=0.78\textwidth]{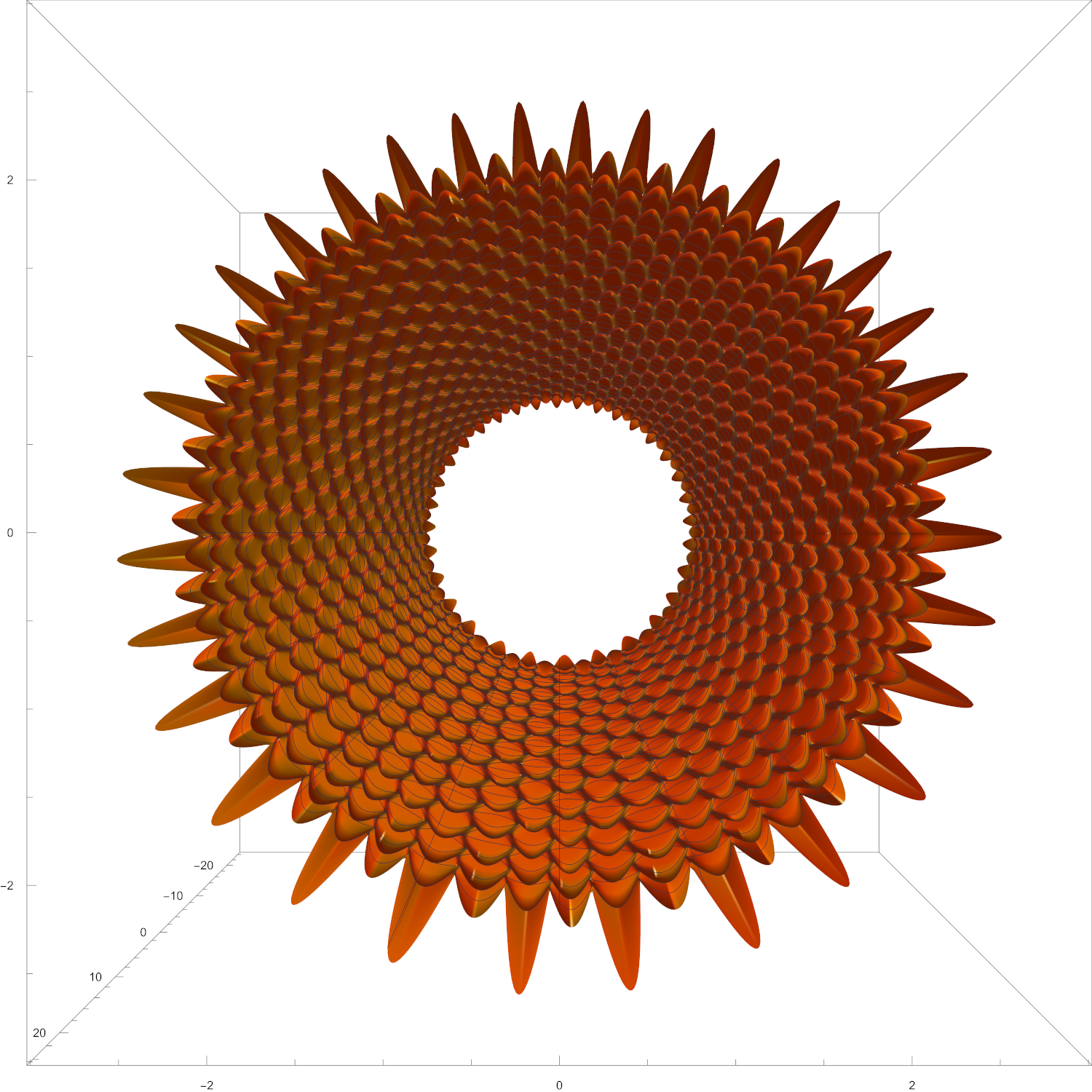}
\caption{Case 2.}
\label{fig:case2}
\end{minipage}
\end{figure}

\appendixtitleon
\appendixtitletocon
\begin{appendices}{\clearpage}

\section{Canonical Operator with Violated Quantization Conditions}
\label{pr-co-def}
\subsection{Preliminary Definitions}
Let $\Lambda = \{p=P(\alpha), x=X(\alpha)\}$ be a~Lagrangian torus in the standard symplectic space $\mathbb R^{2n}_{p,x}$, where $\alpha = (\alpha^1,\dots,\alpha^n)$, $\alpha^j\ {\rm mod}\, 2\pi$. Let us endow $\Lambda$ with a~nondegenerate $n$-form $d\mu = \mu(\alpha) d\alpha^1\wedge \dots\wedge d\alpha^n$, fix a~point $\alpha=\alpha_0\in\Lambda$, and define the argument of the Jacobian $\mathscr J^{\varepsilon}(\alpha_0)$, where
$$
\mathscr J^\varepsilon (\alpha) = \frac{d(X_1-i\varepsilon P_1)\wedge \dots \wedge d(X_n-i\varepsilon P_n)}{d\mu} = \frac{1}{\mu(\alpha)} {\rm det}\, [X_\alpha - i\varepsilon P_\alpha]
$$
for all $\varepsilon\in [0,1]$.
\subsection{Local Canonical Operator}
For a~fixed chart $U\subset\Lambda$ (a~simply connected domain) we define a~local canonical operator via nondegenerate phase function. It is a~function $\Phi(x,\theta)$ in a~simply connected domain $V\subset \mathbb R^{n+m}$ such that $d\Phi_{\theta^1}, \dots, d\Phi_{\theta^m}$ are linearly independent everywhere on the set $C_\Phi = \{ (x,\theta)\in U: \Phi_{\theta^1} = \dots = \Phi_{\theta^m} = 0\}$, see \cite{DNSh17}.
Then $C_\Phi$ is a~smooth $n$-manifold, and a~map $j_\Phi:V\to\mathbb R^{2n}_{p,x}$ is a~local diffeomorphism between $C_\Phi$ and $\Lambda_\Phi = j_\Phi(C_\Phi)$. If $\Lambda_\Phi = U$, then the chart $U$ is called a~$\Phi$-chart. Let us choose a~path $\gamma: [0,1]\to\Lambda$, which connects $\gamma(0) = \alpha_0$ and some point $\gamma(1)\in U$. For $A\in C_0^\infty(U)$ we now define an~action of canonical operator $K_{(U,\Phi, \gamma)}A$ associated with $U$, $\Phi$ and $\gamma$. To do that, define a~function
$$
F_{\Phi, d\mu} = \frac{d\mu\wedge(-d\Phi_\theta)}{dx\wedge d\theta}
$$
on $C_\Phi$, and
$$
{\rm arg}\, F_{\Phi, d\mu} = -\frac{\pi m}{2} - {\rm arg}\, \mathscr J^1(\gamma(1)) - {\rm arg}\, {\rm det}\, \begin{pmatrix} - i\Phi_{xx} & -i\Phi_{x\theta} \\ -i\Phi_{\theta x} & -i \Phi_{\theta\theta}\end{pmatrix},
$$ 
where the last term is understood as a~sum of complex arguments of eigenvalues of the matrix chosen to be in the interval $[-\pi/2, \pi/2]$. Let us now define
$$
K_{(U,\Phi,\gamma)} A = \frac{e^{\frac{\pi i m}{4}}}{(2\pi h)^{m/2}} e^{\frac{i}{h}s_U - \frac{\pi i}{2} m_U} \int_{\mathbb R^m} e^{\frac{i}{h}\Phi(x,\theta)} \sqrt{|F_{\Phi, d\mu}|} \widetilde A(x,\theta)\, d\theta,
$$
where $s_U = \int_\gamma P(\alpha)dX(\alpha) - \Phi(j_\Phi^{-1}(\gamma(1)))$, $m_U = -\frac{1}{\pi} {\rm arg}\, F_{\Phi, d\mu}$, and $\widetilde A$ is an arbitrary smooth function such that $j_{\Phi}\circ\widetilde A\big|_{C_\Phi}=A$.

\subsection{Quantization Conditions and Global Canonical Operator}
Now we define a~global canonical operator $K_{\Lambda, d\mu, \alpha_0, \mathscr J^\varepsilon (\alpha_0)}$ associated with the Lagrangian manifold $\Lambda$, which is endowed with $d\mu$, $\alpha_0$ and $\mathscr J^\varepsilon (\alpha_0)$. 

From now on we will assume that $\Lambda$ is a torus.
Let $\Gamma_1, \dots,\Gamma_n$ be basis cycles on $\Lambda$ corresponding to coordinates $\alpha^1,\dots,\alpha^n$ (modulo $2\pi$) respectively. Define 
\begin{equation}
\label{QC-def}
q_j = \frac{1}{4} {\rm ind}_{\Gamma_j} -\frac{1}{2\pi} \int_{\Gamma_j} P(\alpha) dX(\alpha), \qquad \text{where} 
\end{equation}
${\rm ind}_{\Gamma_j} = \underset{\varepsilon \to +0}{\rm lim}\, \Delta {\rm arg}_{\Gamma_j} \mathscr J^\varepsilon(\alpha)$ is the Maslov index, see \cite{Arnold67}. The standard quantization conditions $q_1,\dots,q_n\equiv 0$ ensure that the global canonical operator $K_{\Lambda,d\mu,\alpha_0,\mathscr J^\varepsilon(\alpha_0)}$ is well-defined on $C^\infty(\Lambda)$.

We generalize this construction by getting rid of quantization conditions. Let $\mathcal A$ be a~smooth function on $\overline \Lambda = \mathbb R^n$ (the universal covering of $\Lambda$) such that for all $\alpha$ and $j$: 
\begin{equation}
\label{QCond-Gen}
\mathcal A (\alpha^1,\dots,\alpha^j+2\pi,\dots, \alpha^n) = e^{2\pi i q_j} \mathcal A(\alpha^1,\dots,\alpha^n).
\end{equation} 
(In other words, $\mathcal A$ is $q$-Bloch for $q$ defined by \eqref{QC-def}).

Let us fix an~atlas of $\Phi$-charts $\Lambda = \bigcup\limits_{j=1}^N U_j$. For each $U_j$ we choose the phase function $\Phi_j$ defined on $U_j$, along with a~path $\gamma_j$. We also fix one of the points on $\overline \Lambda$ that projects onto $\alpha_0$ (by specifying $\alpha\in \mathbb{R}^n$ such that $\alpha \equiv \alpha_0$ (mod $2\pi$)). Denote by $A_j$ a~smooth function on $U_j$ defined by continuing $\mathcal A$ along the path $\gamma_j$ lifted to $\overline\Lambda$. Finally, let $1 = \sum\limits_{j=1}^N \chi_j$ be a~smooth partition of unity on $\Lambda$ such that ${\rm supp}\, \chi_j \subset U_j$. Let us now define
\begin{equation}
\label{COdef}
K_{\Lambda, d\mu, \alpha_0,\mathscr J^1(\alpha_0)}\mathcal A \equiv K_\Lambda \mathcal{A} = \sum_{j=1}^N K_{(U_j,\Phi_j,\gamma_j)} [\chi_j A_j].
\end{equation}
\begin{prop}
\label{CO-QC-gen}
Given conditions \eqref{QCond-Gen}, the expression $K_\Lambda \mathcal A$ defined by \eqref{COdef} is independent modulo $O(h)$ of the choice of the atlas $\bigcup_{j=1}^N U_j$, partition of unity $\{\chi_j\}$, phase functions $\Phi_j$ and paths $\gamma_j$.
\end{prop}
\begin{proof}
It is sufficient to prove that for each pair of triples $(U_1,\Phi_1,\gamma_1)$ and $(U_2,\Phi_2,\gamma_2)$ such that $U_1\bigcap U_2 \ne \emptyset$, and $\chi_0\in C_0^\infty(U_1\bigcap U_2)$ there holds $K_{(U_1,\Phi_1,\gamma_1)}[\chi_0 A_1]\asymp K_{(U_2,\Phi_2,\gamma_2)} [\chi_0 A_2]$ meaning that
\begin{equation}
\label{COtoprove}
K_{(U_1,\Phi_1,\gamma_1)}[\chi_0 A_1] = K_{(U_2,\Phi_2,\gamma_2)} [\chi_0 A_2 + O(h)].
\end{equation}

First, we show that $K_{(U_1,\Phi_1,\gamma_1')}A_1' \asymp K_{(U_1,\Phi_1,\gamma_1'')} A_1''$. Indeed, using \cite[Lemma 2]{DNSh17} by which 
$K_{(U_1,\Phi_1,\gamma_1')}A_1' \asymp K_{(U_1,\Phi_1,\gamma_1\sqcup\gamma_{10})}\widetilde A_1'$, where $\gamma_{10}$ is a path lying in $U_1$ and connecting $\gamma_1'(1)$ with $\gamma_1''(1)$, and the fact that the branches of $\mathcal{A}$ associated with the paths $\gamma_1'$ and $\gamma_1\sqcup\gamma_{10}$ are equal, we can restrict ourselves to the case $\gamma_1'(1)=\gamma_1''(1)$. 

Now let us expand contour $\Gamma:=\gamma_1'\sqcup (\gamma_1'')^{-1}$ into basis cycles: $\Gamma = \sum_{j=1}^n k_j \Gamma_j$ for $k\in\mathbb{Z}^n$, then it follows from \eqref{QCond-Gen} that
\begin{equation}
\label{A1qc}
A_1' = e^{2\pi i \langle k, q\rangle} A_1''= e^{-\frac{i}{h}\oint_\Gamma p\, dx + \frac{\pi i}{2}{\rm ind}_\Gamma} A_1''.
\end{equation}
According to \cite[Lemma~3]{DNSh17}, $K_{(U_1, \Phi_1, \gamma_1')} A_1'' \asymp e^{\frac{i}{h}\oint_\Gamma pdx - \frac{\pi i}{2}{\rm ind}_\Gamma}
K_{(U_1, \Phi_1, \gamma_1'')} A_1''$, which yields $K_{(U_1,\Phi_1,\gamma_1')}A_1' \asymp K_{(U_1,\Phi_1,\gamma_1'')} A_1''$.

Going back to \eqref{COtoprove}, let us choose a~path $\gamma_0$ connecting $\alpha_0$ with some point from $U_1\bigcap U_2$. Then, according to what we have proved:
$$
K_{(U_1, \Phi_1, \gamma_1)} [\chi_0 A_1] \asymp K_{(U_1, \Phi_1, \gamma_0)} [\chi_0 A_0], \qquad K_{(U_2, \Phi_2, \gamma_2)} [\chi_0 A_2] \asymp K_{(U_2, \Phi_2, \gamma_0)} [\chi_0 A_0],
$$
where $A_0$ is a~branch of $\mathcal A$ associated with $\gamma_0$. Finally, 
by \cite[Lemma~4]{DNSh17} it follows that
$
K_{(U_1, \Phi_1, \gamma_0)} [\chi_0 A_0] \asymp K_{(U_2, \Phi_2, \gamma_0)} [\chi_0 A_0].
$
Thus, formula \eqref{COtoprove} is proved, and so is Proposition~\ref{CO-QC-gen}.
\end{proof}

\end{appendices}


\begin{thebibliography}{99}

\bibitem{MasFed}
V.~P.~Maslov, M.~V.~Fedoryuk, Semiclassical approximation for the equations of quantum mechanics, Moscow, Nauka, 1976 [in Russian].


\bibitem{Laz}
V.~F.~Lazutkin, KAM Theory and Semiclassical Approximations to Eigenfunctions, Berlin, Springer Verlag, 1993.

\bibitem{DN2017}
S.~Yu.~Dobrokhotov, V.~E.~Nazaikinskii, On the Asymptotics of a Bessel-Type Integral Having Applications in Wave Run-Up Theory, Mat. Zametki, 102:6 (2017), 828–835; Math. Notes, 102:6 (2017), 756–762


\bibitem{ADNTs19}
A.~Yu.~Anikin, S.~Yu.~Dobrokhotov, V.~E.~Nazaikinskii, A.~V.~Tsvetkova, “Uniform asymptotic solution in the form of an Airy function for semiclassical bound states in one-dimensional and radially symmetric problems”, TMF, 201:3 (2019), 382–414; Theoret. and Math. Phys., 201:3 (2019), 1742–1770

\bibitem{DN19}
S.~Yu.~Dobrokhotov, V.~E.~Nazaikinskii, Efficient formulas for the Maslov canonical operator near a simple caustic, Russ. J. Math. Phys. 25, 545–552 (2018). 

\bibitem{ADNTs20}
A.~Yu.~Anikin, S.~Yu.~Dobrokhotov, V.~E.~Nazaikinskii, A.~V.~Tsvetkova, Asymptotic eigenfunctions of the operator $\nabla D(x)\nabla$ defined in a two-dimensional domain and degenerating on its boundary and billiards with semi-rigid walls, Diff Equat 55, 644–657 (2019).


\bibitem{ADNTs21}
A.~Yu.~Anikin, S.~Yu.~Dobrokhotov, V.~E.~Nazaikinskii, A.~V.~Tsvetkova, Nonstandard Liouville tori and caustics in asymptotics in the form of Airy and Bessel functions for two-dimensional standing coastal waves, Algebra i Analiz, 2021, Volume 33,	Issue 2,	Pages 5--34.

\bibitem{Naz12}
V.~E.~Nazaikinskii, Phase space geometry for a wave equation degenerating on the boundary of the domain, Math Notes, 92, 144–-148 (2012).

\bibitem{Naz14}
V.~E.~Nazaikinskii, The Maslov canonical operator on Lagrangian manifolds in the phase space corresponding to a wave equation degenerating on the boundary. Math Notes, 06, 248--260 (2014).




\bibitem{Urs}
F.~Ursell, Edge waves on a sloping beach, Proc. R. Soc. Lond. A. 214 (1952), 79–97. MR 0050420.

\bibitem{Zhev}
P.~N.~Zhevandrov, Edge waves on a gently sloping beach: uniform asymptotics, J. Fluid Mech. 233
(1991), 483–493. MR 1140088.

\bibitem{MerZ}
A.~Merzon and P.~Zhevandrov, High-frequency asymptotics of edge waves on a beach of nonconstant
slope SIAM J. Appl. Math. 59 (1998), no. 2, 529–546. MR1654415.



\bibitem{AnDob}
A.~Yu.~Anikin, S.~Yu.~Dobrokhotov, Diophantine Tori and Pragmatic Calculation of Quasimodes for Operators with Integrable Principal Symbol, Russ. J. Math. Phys, 27:3, 2020, 299--308.

\bibitem{LazQuas}
V.~F.~Lazutkin, Quasiclassical asymptotic behavior of eigenfunctions, Partial differential equations – 5, Itogi Nauki i Tekhniki. Ser. Sovrem. Probl. Mat. Fund. Napr., 34, VINITI, Moscow, 1988, 135--174.

\bibitem{AniRyk22}
A.~Yu.~Anikin, V.~V.~Rykhlov, Constructive Semiclassical Asymptotics of Bound States of Graphene in a Constant Magnetic Field with Small Mass, Math. Notes, 2022, 111:2, 173--192.

\bibitem{Matv}
V.~S.~Matveev, Asymptotic eigenfunctions of the operator $\nabla D(x, y)\nabla$ that correspond to Liouville
metrics and waves on water trapped by bottom irregularities, Mat. Zametki 64 (1998), no. 3, 414–
422; English transl., Math. Notes 64 (1998), no. 3–4, 357–363. MR 1680205.



\bibitem{OleyRadk}
O.~A.~Oleinik, E.~V.~Radkevich, Second order equations with nonnegative characteristic form, Itogi Nauki. Ser. Matematika. Mat. Anal. 1969, VINITI, Moscow, 1971, 7–252

\bibitem{DN-Uni}
S.~Yu.~Dobrokhotov, V.~E.~Nazaikinskii, Uniformization of Equations with Bessel-Type Boundary Degeneration and Semiclassical Asymptotics, Mat. Zametki, 107:5 (2020), 780–786; Math. Notes, 107:5 (2020), 847--853


\bibitem{Arnold67}
V.~I.~Arnold, Characteristic class entering in quantization conditions, Funct. Anal. Appl., 1:1 (1967), 1--13

\bibitem{DNSh17}
S.~Yu.~Dobrokhotov, V.~E.~Nazaikinskii, A.~I.~Shafarevich, New integral representations of the Maslov canonical operator in singular charts, Izv. RAN. Ser. Mat., 81:2 (2017), 53–96; Izv. Math., 81:2 (2017), 286--328

\bibitem{DNSh21}
S.~Yu.~Dobrokhotov, V.~E.~Nazaikinskii, A.~I.~Shafarevich, Efficient asymptotics of solutions to the Cauchy problem with localized initial data for linear systems of differential and pseudodifferential equations, Uspekhi Mat. Nauk, 76:5(461) (2021), 3–80; Russian Math. Surveys, 76:5 (2021), 745–819


\bibitem{Unif23}
A.~Y. Anikin, S.~Y. Dobrokhotov, V.~E. Nazaikinskii, A.~A. Tolchennikov, Uniformization and Semiclassical Asymptotics for a Class of Equations Degenerating on the Boundary of a Manifold. J Math Sci 270, 507–530 (2023).

\end{thebibliography}
\end{document}